\documentclass[12pt,a4paper]{amsart} 
\usepackage{amsmath}
\usepackage{amsthm}
\usepackage{amssymb}
\usepackage{mathrsfs}
\usepackage{fancyhdr}
\usepackage{afterpage}
\newtheorem{thm}{Theorem}[section]
\newtheorem{lem}[thm]{Lemma}
\newtheorem{cor}[thm]{Corollary}
\newtheorem{prop}[thm]{Proposition}

\theoremstyle{remark}
\newtheorem{rem}[thm]{Remark}
\newtheorem*{rem*}{Remark}

\theoremstyle{definition}
\newtheorem{dfn}[thm]{Definition}
\newtheorem{ex}[thm]{Example}

\numberwithin{equation}{section}

\newcommand{\om}{\Omega}

\newcommand{\ch}{{\mathcal{O}_c}}
\newcommand{\oa}{{\mathcal{O}_c^\mathrm{a}}}
\newcommand{\C}{\mathbb{C}}

\newcommand{\PP}{\mathbb{P}}
\newcommand{\Ker}{\operatorname{Ker}}

\begin{document}
\afterpage{\cfoot{\thepage}}
\clearpage

\title{On the growth exponent of c-holomorphic functions with algebraic graphs}

\author{Adam Bia\l o\.zyt, Maciej P. Denkowski, Piotr Tworzewski, Tomasz Wawak}\address{Jagiellonian University, Faculty of Mathematics and Computer Science, Institute of Mathematics, \L ojasiewicza 6, 30-348 Krak\'ow, Poland}\email{adam.bialozyt@doctoral.uj.edu.pl}\email{maciej.denkowski@uj.edu.pl}\email{piotr.tworzewski@uj.edu.pl}\email{tomasz.wawak@student.uj.edu.pl}\date{February 8th 2014, {\it Extended:} December 22nd 2019}

\begin{abstract} 
This paper is the first of a series dealing with c-holomorphic functions defined on algebraic sets and having algebraic graphs. These functions may be seen as the complex counterpart of the recently introduced \textit{regulous} functions. Herein we study their growth exponent at infinity. A general result on injectivity on fibres of an analytic set together with a theorem of Tworzewski and Winiarski gives a bound for the growth exponent of a c-holomorphic function with algebraic graph in terms of the projective degrees of the sets involved. We prove also that algebricity of the graph is equivalent to the function being the restriction of a rational function (a Serre-type theorem). Then we turn to considering generically finite c-holomorphic mappings with algebraic graphs and we prove a B\'ezout-type theorem. We also study a particular case of the \L ojasiewicz inequality at infinity in this setting.
\end{abstract}

\maketitle

\section{Introduction}

The main object we will be dealing with in this article are continuous functions with algebraic graphs, defined on a given algebraic subset of a complex finite-dimensional vector space $M$. Let us stress from the beginning that this is a much larger class than the usual class of regular functions (see Example \ref{przyklad} and Theorem \ref{3.4}). It seems that this particular class of functions has not been yet studied even though it is part of Remmert's larger class of c-holomorphic functions.

To make the notation clearer, we will consider $M={\C}^m$ (anyway, everything here is invariant under linear isomorphisms). Our general aim is to present several effective results concerning this class of functions and thus we study their growth exponent and --- in a forthcoming paper --- also their \L ojasiewicz's exponents at infinity, the Nullstellensatz in this class. Actually, the class of functions we are interested in coincides with the class of c-holomorphic functions with algebraic graphs. C-holomorphic functions were introduced by R. Remmert in his work on proper projections of analytic sets as continuous functions on defined analytic sets and holomorphic at regular points. When dealing with such functions we are bound to use purely geometric methods, as they do not enjoy good enough differential or algebraic properties (compare e.g. \cite{D}). The class of c-holomorphic functions lies in between the class of strongly holomorphic functions (local restrictions of holomorphic functions in the ambient space) and Cartan's weakly holomorphic functions (functions defined and holomorphic in regular points, locally bounded near the singularities). 

It may be interesting to note a kind of analogy between c-holomorphic functions with algebraic graphs and the recently introduced \textit{regulous} functions \cite{FHMM} as well as in \cite{Kol} (see also \cite{KolN}), in connection with the important results of \cite{K}. In essence, regulous functions are real rational functions that admit continuous extensions. Similarly, as shown in Theorem \ref{3.4}, c-holomorphic functions with algebraic graphs on an algebraic set are exactly continuous restrictions of rational functions. 

\medskip
Our main results presented here include a general theorem concerning the generic injectivity of non-constant c-holomorphic functions on fibred analytic sets (Theorem \ref{genInj}), an estimate of the growth exponent at infinity of c-holomorphic functions with algebraic graph (Theorem \ref{oszacowanie} and Proposition \ref{rwn}), an algebraic graph theorem (Theorem \ref{3.4}), a general B\'ezout-type inequality (Theorem \ref{Bezout} and Theorem \ref{krzywa} for the case of curves) and a new sharp upper bound for the \L ojasiewicz exponent at infinity of a proper c-holomorphic mapping with algebraic graph (Theorem \ref{AB}).

\medskip
For the convenience of the reader we recall the definition of a c-holo\-morph\-ic mapping. Let $A\subset\om$ be an analytic subset of an open set $\om\subset{\C}^m$.

\begin{dfn}[R. Remmert, see \cite{L}, \cite{Wh}] A mapping $f\colon
A\to{\C}^n$ is called {\it c-ho\-lo\-morph\-ic} if it is
continuous and the restriction of $f$ to the subset of regular
points ${\rm Reg} A$ is holomorphic. We denote by $\ch (A,{\C}^n)$
the ring of c-holomorphic mappings (with pointwise multiplication), and by $\ch (A)$ the ring of
c-holomorphic functions. 
\end{dfn}
It is a way of generalizing the notion of holomorphic mapping onto
sets having singularities and a more convenient one than the usual notion of {\it weakly holomorphic functions} (i.e. functions defined and holomorphic on $\mathrm{Reg} A$ and locally bounded on $A$). The following theorem is fundamental for all what we shall do (cf. \cite{Wh} 4.5Q):

\begin{thm}\label{graf} A mapping $f\colon A\to {\C}^n$
is c-holomorphic iff it is continuous and its graph
$\Gamma_f:=\{(x,f(x))\mid x\in A\}$ is an analytic subset of
$\om\times{\C}^n$.\end{thm}

For a more detailed list of basic properties of c-holomorphic mappings see \cite{Wh} and \cite{D}. 
 Let us just note that composing two c-holomorphic functions yields again a c-holomorphic function:
\begin{prop}\label{zlozenie}
Assume that $X$ and $Y$  are analytic subsets of open sets $\Omega\subset{\C}^m$ and $D\subset{\C}^n$ respectively, and $f\colon X\to Y$ and $g\colon Y\to {\C}$ are c-holomorphic. Then $g\circ f\in\mathcal{O}_c(X)$.
\end{prop}
\begin{proof}
Since $g\circ f$ is continuous, it is enough to check that its graph is locally analytic (cf. Theorem \ref{graf}). Consider the natural projection $\pi(x,y,z)=(x,z)$ where $(x,y,z)\in {\C}^m\times{\C}^n\times{\C}$. Then 
$$
\Gamma_{g\circ f}=\pi((\Gamma_f\times{\C})\cap (X\times\Gamma_g)).
$$
The projected set $Z:=(\Gamma_f\times{\C})\cap (X\times\Gamma_g)$ is analytic. 
Given a compact set $K\subset X\times{\C}$ we put $L:=p(K)$ for $p(x,z)=x$.
Now, $L$ is compact and 
$$
\pi^{-1}(K)\cap Z\subset L\times f(L)\times g(f(L)),
$$
hence $\pi|_Z$ is proper. By the Remmert Proper Mapping Theorem, $\Gamma_{g\circ f}$ is analytic as required.
\end{proof}

Finally, we recall some notions we will be using. For a polynomial $P\in\mathbb{C}[x_1,\dots, x_m]$ let $P^+$ denote its homogeneous part of maximal degree, i.e. $\deg P^+=\mathrm{deg} P$ and $\deg (P-P^+)<\deg P$. We denote by $\tilde{P}(t,z)=\sum_{|\alpha|\leq d}a_{\alpha} z^\alpha t^{d-|\alpha|}$ the homogenization of $P(z)=\sum_{|\alpha|\leq d} a_{\alpha} z^\alpha$ where $d=\deg P$. Then, $\tilde{P}(0,z)=P^+(z)$. 

Recall also (see \cite{L} VII.\S 7 and \cite{Ch}) that if $\Gamma\subset{\C}^n$ is algebraic of pure dimension $k$, then its projective degree $\deg \Gamma=\#(L\cap \Gamma)$ for any $L\subset{\C}^n$ affine subspace of dimension $n-k$ transversal to $\Gamma$ and such that $L_\infty\cap\overline{\Gamma}=\varnothing$, where $\overline{\Gamma}=\overline{\PP(\{1\}\times \Gamma)}$, for $\PP\colon{\C}^{1+n}_*\to\PP_n$ the canonical projection, is the projective closure and $L_\infty$ denotes the points of $L$ at infinity (i.e. the intersection of $\overline{L}$ with the hyperplane at infinity in $\mathbb{P}_n$). The point is that the condition  $L_\infty\cap\overline{\Gamma}=\varnothing$ (equivalently: $L_\infty\cap{\Gamma}_\infty=\varnothing$) is equivalent to the inclusion
$$
\Gamma\subset\{u+v\in L'+L\mid |v|\leq\mathrm{const.} (1+|u|)\}
$$
where $L'$ is any $k$-dimensional affine subspace such that $L'+L={\C}^n$. Moreover, for any $(n-k)$-dimensional affine subspace $L$ cutting $A$ in a zero-dimensional set (with no additional hypotheses) there is $\#(L\cap\Gamma)\leq\deg \Gamma$. 
Writing $G'_{n-k}({\C}^n)$ for the set of affine hyperplanes of dimension $n-k$, we have
$$
\deg \Gamma=\max\{\#(L\cap \Gamma)\mid L\in G'_{n-k}({\C}^n)\colon \dim (L\cap \Gamma)=0\}.
$$ 
The projective degree of $\Gamma$ is in fact equal to the local degree (Lelong number) at zero of the cone defined by $\Gamma$. Namely, if $C_\Gamma$ denotes the closure in ${\C}\times{\C}^n$ of the pointed cone $S_\Gamma:=\{{\C}_*\cdot (1,x)\mid x\in \Gamma\}$, then $\deg \Gamma={\deg}_0 C_\Gamma$. Of course, $C_\Gamma=\PP^{-1}(\overline{\Gamma})\cup\{0\}$.


Let $\Gamma^*:=\{x\in{\C}^n\mid (0,x)\in C_\Gamma\}$ which corresponds to the points at infinity but seen in the affine space. In particular, $\{P^+=0\}=\{P=0\}^*$. Moreover, $\Gamma^*=\{v\in{\C}^n\mid \exists \Gamma\ni v_\nu\to \infty, \lambda_\nu\in{\C}\colon \lambda_\nu v_\nu\to v\}$ and this is a complex cone. Finally, observe that now $V_\infty\cap W_\infty=\varnothing$, where $V,W\subset{\C}^n$ are algebraic, is equivalent to $V^*\cap W^*=\{0\}$.
Note by the way that with this approach it is easy to check 
the following:
\begin{prop}\label{stopien wykresu}
 Let $F\colon{\C}^m\to{\C}$ be a non-constant polynomial. Then $\deg F=\deg\Gamma_F.$
\end{prop}
\begin{proof}
Observe that $\Gamma_F$ is the zero set of the polynomial $G(x,t)=t-F(x)$ of degree $d:=\deg F$ (we may assume that $d>1$, otherwise there is nothing to prove). Clearly, $G^+=-F^+$ and so $\Gamma_F^*=\{x\colon F^+(x)=0\}\times{\C}$. Therefore, an affine line $L\in G_1'({\C}^{m+1})$ realizing $\deg\Gamma_F$ can be chosen of the form $\ell\times\{t_0\}$ with $\ell\in G_1'({\C}^m)$ such that $\ell^*\cap\{F^+=0\}=\varnothing$. But then $L\cap \Gamma_F$ corresponds to the degree of $F|_\ell$ which is $d$.
\end{proof}
\begin{rem}
 It is worth noting that such an equality is no longer true for polynomial mappings (where the degree is defined to be the greatest degree of the components, which is accounted for in the next section). Indeed, it suffices to consider a proper polynomial mapping $F=(F_1,\dots, F_m)\colon{\C}^m\to{\C}^m$ such that $\deg F_j>1$ and $\bigcap\{F^+_j=0\}=\{0\}$. Then by the B\'ezout Theorem, for the generic $w\in{\C}^m$ (\footnote{by `generic' we always mean `apart from a nowheredense algebraic set'.}) there is $\#F^{-1}(w)=\prod\deg F_j>\deg F$ and since $F^{-1}(w)=\Gamma_F\cap({\C}^m\times\{w\}^m)$, it follows that $\deg\Gamma_F\geq \#F^{-1}(w)$. For instance, $m=2$, $F(x,y)=(x^2,y^3)$ is a good example.

Nonetheless, if $F\colon {\C}_x^m\to{\C}_y^n$ is a proper polynomial mapping (hence $n\geq m$), then in view of the main result of \cite{TW1} applied to the natural projection $\pi(x,y)=x$ restricted to $\Gamma_F$ (it is obviously one-sheeted), we obtain
$$
\Gamma_F\subset\{(x,y)\in{\C}^m\times{\C}^n\mid |y|\leq C(1+|x|)^{\deg\Gamma_F}\}
$$
with some $C>0$ (here $|\cdot|$ denotes any norm). This means that $|F(x)|\leq C(1+|x|)^{\deg\Gamma_F}$ and so $\mathrm{deg} F\leq\deg\Gamma_F$ (cf. section 3).
\end{rem}

\section{Generic injectivity of c-holomorphic non-constant mappings on fibred analytic sets}

We consider the following general situation. 
Let $A\subset D\times{\C}^n$ be a pure $k$-dimensional analytic set with proper projection $\pi(x,y)=x$ onto the domain $D\subset{\C}^k$. Let $d$ denote the multiplicity of $\pi|_A$ as a branched covering and $\sigma_\pi\subsetneq D$ its critical (or discriminant) set. 

We assume throughout this section that for projections $\pi'$ close to $\pi$ we still have $\pi'|_A$ a $d$-sheeted branched covering. Close, in this case, means close in the space of all epimorphisms (projections, as it were) ${\C}^{k+n}\to{\C}^k$. Equivalently, that means that $\Ker\pi$ and $\Ker\pi'$ are close in the $n$-th Grassmannian $G_n({\C}^{k+n})$, i.e. in the sense of the Kuratowski convergence, cf. \cite{DP} where this is explained in details.  We shortly recall that thefamily of closed subset $\mathcal{F}_X$ of a locally compact metric space $X$ can be endowed with a metrizable and compact topology in which the convergence of closed sets $F_\nu\to F$ is equivalent to \begin{enumerate}
\item[(a)] Any $x\in F$ is the limit of a sequence $F_\nu\ni x_\nu\to x$;
\item[(b)] Given a compact set $K\subset X\setminus F$, there is $F_\nu\cap K=\varnothing$ for all $\nu$ large enough.
\end{enumerate}
This is a natural generalization of the convergence of compact sets in the Hausdorff metric. More importantly, the natural topology of the Grassmannian and of $G'_k({\C})$ gives exactly this convergence (cf. \cite{L} and \cite{DP}). It is then easy to check that the mapping
$$
{\C}^m\times G_k({\C}^m)\ni (x,L)\mapsto x+L\in G'_k({\C}^m)
$$
is continuous.

As a matter of fact, we are particularly interested in the following two situations:

\smallskip
\begin{enumerate}                                                                                  \item Suppose that $D={\C}^k$ and $A$ is algebraic. If $\pi$ is a projection realizing the projective degree $\deg A$, then $A\subset\{(x,y)\mid |y|\leq \mathrm{const.}(1+ |x|)\}$. Thence it is easy to see that for $n$-dimensional linear subspaces $L\in G_n({\C}^{k+n})$ close enough to $L_0:=\{0\}^k\times{\C}^n$ the projection $\pi^L$ along $L$ is still proper on $A$ and has multiplicity $\deg A$ (\footnote{Indeed, for $L$ close to $L_0$ we will still have $A\subset \{x+u\mid x\in {\C}^k, u\in L, |u|\leq\mathrm{const.} (1+|x|)\}$; this follows from the fact that this inclusion is equivalent to the intersection at infinity being void: $A_\infty\cap \overline{L_0}=\varnothing$.}).

\smallskip
\item Suppose that $A\subset{\C}$ is locally analytic, $0\in A$ and $d$ is the local degree (Lelong number) ${\deg}_0 A$. If $\pi|_A$ realizes ${\deg}_0 A$ (as its multiplicity), then for the tangent cone $C_0(A)$, keeping the notations introduced so far, we have $L_0\cap C_0(A)=\{0\}$ and this property is open in the Grassmannian (cf. \cite{Ch} or \cite{L} (\footnote{By \cite{L} B.6.8, for any $L_0\in G_n({\C}^{k+n})$, the sets $\{L\in G_n({\C}^{k+n})\mid L\subset\{u+v\in L_0+L_0^\perp\mid |v|\leq c|u|\}\}$ for varying $c>0$ form a basis of neighbourhoods of $L_0$.})). Since there is a bounded neighbourhood $W\subset{\C}^n$ of zero such that $(\{0\}^k\times \overline{W})\cap A=\{0\}$, by taking $A\cap (G\times W)$, for some bounded neigbourhood $G\subset{\C}^k$ of the origin, instead of $A$, we may assume that $L_0\cap A=\{0\}$. Note that $L_0\cap A=\{0\}$ implies $A\cap (U\times {\C}^n)=A\cap( U\times V)$ for some bounded neighbourhoods of the origin $U\subset{\C}^k$, $V\subset{\C}^n$. It is easy to see that we may choose the same $U$ for any vector complement $L$ of ${\C}^k\times\{0\}^n$ sufficiently close to $L_0$ (\footnote{This is a simple consequence of the observation that $L_0\cap C_0(A)=\{0\}$ is equivalent to saying that $A\cap (U\times V)\subset\{(x,y)\mid |y|\leq \mathrm{const.}|x|\}$ for some neighbourhoods of zero. In our situation we can take $V={\C}^n$, as already noted. The form of neighbourhoods of $L_0$ (see the previous footnote) account for the rest.}), i.e. in such a way that $\pi^L(A\cap (U+L))=U$. Since $U$ may be chosen connected, we put $D:=U$. 
                                                                                                           
\end{enumerate}

Accordingly with the notations above, we will identify the projections $\pi^L$ with their kernels $L=\Ker\pi^L$. Write $L_0=\Ker\pi$ and $L'=\Ker\pi'$, if necessary.

\bigskip
Let $f\colon A\to {\C}$ be a c-holomorphic function which is \textit{non-constant on any irreducible component of} $A$. 

 Let $\mathscr{P}\subset G_n({\C}^{k+n})$ denote the neighbourhood of $\pi$ for which the multiplicity of the projections onto $D$ is $d$. 

\begin{thm}\label{genInj}
Under the assumptions made above, there is a projection $\pi'\in\mathscr{P}$ arbitrarily close to $\pi$ and such that $f$ is injective on the generic fibre of $\pi'|_A$.
\end{thm}
\begin{proof}
We will show that in any neighbourhood of $L_0$ in $G_n({\C}^{k+n})$, there is an $L$ such that 
$$
\forall x\in D\setminus\sigma_L,\> \#f(A\cap (x+L))=d,\leqno{(\#)}
$$
for some analytic set $\sigma_L\subsetneq D$ (\footnote{Note that in both situations (1) and (2) described before the Theorem it is possible to take a common $D$ for all $L$ close enough to $L_0$.}). If we denote $\Gamma_L=(\pi^L\times \mathrm{id}_{\C})(\Gamma_f)$ and $p(x,t)=x$, then condition $(\#)$ means precisely that $p|_{\Gamma_L}$ has multiplicity $d$ (it could only be less) as a branched covering. From this we infer that in $(\#)$ it is enough to find one appropriate point $x$. 

Let $\sigma_{\pi^L}$ denote the critical set of the branched covering $\pi^L|_A$. For a given $L$ near $L_0$ put $$Z_L=\{x\in D\mid \#f(A\cap (x+L))<d\}.$$ Of course, there is always $\sigma_{\pi^L}\subset Z_L$.

If for $L_0$ we have $Z_{L_0}=D$, then given a point $x_0\in D\setminus\sigma_{\pi^{L_0}}$ and its (simply) connected neighbourhood $V\subset D\setminus\sigma_{\pi^{L_0}}$ for which $(\pi^{L_0})^{-1}(V)\cap A=\bigcup_1^d \Gamma_j$ is a disjoint union of graphs $\Gamma_j$ of holomorphic functions $\gamma_j$, there are indices $i\neq j$ such that 
$$\forall x\in V,\ f(x,\gamma_i(x))=f(x,\gamma_j(x)).$$ 
This follows easily from the identity principle. Indeed, write $g_i(x):=f(x,\gamma_i(x))$, $x\in V$. These are holomorphic functions and so $V_{ij}:=\{x\in V\mid g_i(x)=g_j(x)\}$ are analytic. By assumptions $\bigcup_{1\leq i<j\leq d} V_{ij}=V$ and so at least one of the sets $V_{ij}$ must coincide with $V$. 

Let us suppose hereafter that for all $L$ in a neighbourhood $\mathscr{P}'\subset\mathscr{P}$ of $\pi$ we have $Z_L=D$. 

Of course, $V$ and $\Gamma_j$ depend \textit{a priori} on $L_0$. But observe that the convergence $L\to L_0$ implies the convergence (\footnote{The sets being finite, this convergence coincides with the one in the Hausdorff measure.}) $A\cap (x_0+L)\to A\cap (x_0+L_0)$ (use for instance \cite{TW2}). In particular, $\#A\cap(x_0+L)=d$ when $L$ is close enough to $L_0$. 
The idea is that it should be possible to find a neighbourhood $V_1$ of $x_0$ (necessarily contained in $D\setminus\sigma_{\pi_L}$ since $A$ is a union of $d$ disjoint graphs over it), over which the sets $\Gamma_j$ are still graphs in the direction $L$. 

To be more precise, suppose that we have separated the $d$ points of $A\cap (x_0+L_0)$ by pairwise disjoint balls $B_j$ of radius $\varepsilon$, centred at $(x_0,\gamma_j(x_0))$, so that $B_j\cap A=B_j\cap\Gamma_j$. Using the continuity of the map $(x,L)\mapsto A\cap (x+L)$ at the point $(x_0,L_0)$ which we have from \cite{TW2} (\footnote{The intersection $A\cap (x_0+L_0)$ being proper.}), we can find neighbourhoods $x_0\in V_1\subset V$ and $L_0\in\mathscr{P}_1\subset\mathscr{P}'$ such that for any $(x,L)\in V_1\times\mathscr{P}_1$, the Hausdorff distance between $A\cap (x+L)$ and $A\cap (x_0+L_0)$ does not exceed $\varepsilon$. But then we can write $(\pi^L)^{-1}(x)=\{z^L_1,\dots,z_d^L\}$ which is a set of $d$ pairwise different points numbered consistently according to which ball $B_j$ the point $z_j^L$ belongs to. This implies in particular that $z_j^L\in\Gamma_j$ so that we can define holomorphic  inverses $V_1\ni x\mapsto \gamma^L_j(x)$ to $\pi^L|_{\Gamma_j}$ with graphs $\Gamma_j^L\subset\Gamma_j$, for $j=1\dots,d$, $L\in \mathscr{P}_1$.

The same argument as before gives us for each $L\in\mathscr{P}_1$ two indices $1\leq i_L< j_L\leq d$ such that for any 
$x\in V_1$, $f(x,\gamma_{i_L}^L(x))=f(x,\gamma_{j_L}^L(x)).$ 
Next, we show that the indices can be chosen independent of $L$ sufficiently close to $L_0$. Note that the property of `gluing up' two sheets in this way is closed with respect to $L$, i.e. if $\mathscr{P}_1\ni L_\nu\to L_1\in\mathscr{P}_1$, then for any $x_1\in V_1$, $A\cap(x_1+L_\nu)\to A\cap (x_1+L_1)$ by a similar argument as earlier based on \cite{TW2}, and so again separating the points in the fibre $A\cap (x_1+L_1)$ leads to the conclusion that $(x_1,\gamma_j^{L_\nu}(x_1))\to (x_1,\gamma_j^{L_1}(x_1))$ for any $j=1,\dots,d$. Therefore, if $(i_{L_\nu},j_{L_\nu})=(i,j)$, for all $\nu$, then also $(i_{L_1},j_{L_1})=(i,j)$. Hence, the sets 
$$
\mathscr{P}_{ij}:=\{L\in \mathscr{P}_1\mid \forall x\in V,\ f(x,\gamma^L_i(x))=f(x,\gamma^L_j(x))\}
$$
are closed and, obviously, $\mathscr{P}_1=\bigcup_{1\leq i<j\leq d}\mathscr{P}_{ij}$. The Baire Category Theorem ensures that $\operatorname{int}\mathscr{P}_{ij}\neq\varnothing$, for some $i<j$. Then we find an open subset $\mathscr{P}_1'\subset\mathscr{P}_1$ and a new $L_0'\in\mathscr{P}_1'$. 

According to \cite{L} B.6.8, we may assume that $\mathscr{P}_1'$ is of the form 
$$
\mathscr{P}_1'=\{L\in G_n({\C}^{k+n})\mid L\subset\{u+v\in L'_0+{(L'_0)}^\perp\colon |v|\leq C|u|\}\},
$$
for some $C>0$. Let $a$ and $c$ denote the unique intersection points of $x_0+L_0'$ with $\Gamma_i$ and $\Gamma_j$, respectively, where $i,j$ are the indices chosen above. For positive integers $\nu$ we write $C_\nu$ for the intersection of $\Gamma_j$ with the ball $\mathbb{B}(c,1/\nu)$. 

All we need to show to end the proof is that for some $\nu$ we have $f|_{C_\nu}=f(a)$, for this means that  $f$ is constant on an nonempty open subset of some irreducible component of $A$ and thus, by the identity principle from \cite{D2}, it is constant on that component, contrary to the assumptions.

Suppose that for any $\nu$ there is a point $c_\nu\in C_\nu$ such that $f(c_\nu)\neq f(a)$ which means that for any $L\in\mathscr{P}'_1$, $a-c_\nu\notin L$. Write $a-c_\nu=u_\nu+v_\nu\in L'_0+{(L'_0)}^\perp$. By construction we have the convergence $a-c_\nu\to a-c\in L_0'$ so that $u_\nu\to a-c$, whereas $v_\nu\to 0$. As $a-c\neq 0$, we get, for all $\nu$ large enough,  $|u_\nu|\geq |a-c|/2$ and $C|a-c|/2\geq |v_\nu|$, whence $C|u_\nu|\geq |v_\nu|$ and we may assume this holds for all indices (shifting the sequence, if necessary; note that $C_\nu$ is a nested sequence). 

For a given index $\nu$ let $\ell_\nu$ denote the orthogonal complement of the line ${\C}u_\nu$ in $L_0'$. Then the vector $a-c_\nu=u_\nu+v_\nu$ is orthogonal to $\ell_\nu$. Let $L_\nu:={\C}(a-c_\nu)\oplus \ell_\nu$; we will show that $L_\nu\in\mathscr{P}_1'$. Fix $w\in L_\nu\setminus\{0\}$ and write $w=u+v\in L_0'+(L_0')^\perp$. Then we have a unique representation $u=\alpha u_\nu+u'$ with $\alpha\in{\C}$ and $u'\in\ell_\nu$. On the other hand, we can decompose $w=\beta(a-c_\nu)+u''$ with $\beta\in{\C}$ and $u''\in\ell_\nu$. Therefore, $$\alpha u_\nu+u'+v=\beta(a-c_\nu)+u''$$ together with $a-c_\nu=u_\nu+v_\nu$ leads to
$$
(v-\beta v_\nu)+(\alpha-\beta)u_\nu+(u'-u'')=0
$$
where the three summands are pairwise orthogonal, as they belong to $(L_0')^\perp$, ${\C}u_\nu$ and $\ell_\nu$, respectively. Hence 
$$
v=\beta v_\nu,\quad \alpha=\beta,\quad u'=u''
$$
and so 
$$
C|u|=C\sqrt{|\alpha u_\nu|^2+|u'|^2}\geq C|\alpha u_\nu|\geq |\alpha v_\nu|=|v|
$$
as required for $L_\nu$ to belong to $\mathscr{P}_1'$. But then $a-c_\nu\in L_\nu\in\mathscr{P}_1'$ which is contrary to our assumptions. This ends the proof.
\end{proof}
Let us state clearly what we will need later on:
\begin{cor}\label{injekcja}
Let $A\subset{\C}^m$ be an algebraic irreducible set of dimension $k$ and let $f\in\mathcal{O}_c^\mathrm{a}(A)$. If $f$ is non-constant, then for the generic choice of coordinates in ${\C}^m_z={\C}^k_x\times{\C}^{m-k}_y$ the projection $\pi(x,y)=x$ restricted to $A$ realizes $\deg A$ and $f|_{\pi^{-1}(x)\cap A}$ is injective for the generic $x$.
\end{cor}

In the purely local case we also have the following result that gives a kind of complement to the results of \cite{D}:

\begin{prop}\label{Rzuty}
Let $f\colon (A,0)\to
({\C}_w^k,0)$ be a non-constant c-holo\-morph\-ic germ on a pure
$k$-dimensional analytic germ $A\subset {\C}^m$. Then we can
choose coordinates in ${\C}_z^m={\C}_x^k\times {\C}_y^{m-k}$ in such
a way that for the projections $\pi(x,y)=x$, $\eta:=(\pi\times
\mathrm{id}_{{\C}_w^k})$, $p(x,y,w)=w$, $\varrho(x,w)=w$, $\zeta(x,w)=x$
and the set $\Gamma:=\eta(\Gamma_f)$, we have \begin{enumerate}
\item[{\rm (i)}] $\pi^{-1}(0)\cap C_0(A)=\{0\}$, i.e.
$\mu_0(\pi|_A)={\deg}_0
A$;
\item[{\rm (ii)}] $\mu_0(p|_{\Gamma_f})=\mu_0(\varrho|_\Gamma)$, i.e.
$m_0(f)=\mu_0(\varrho|_\Gamma)$ and so $\mu_0(\eta|_{\Gamma_f})=1$; 
\item[{\rm (iii)}] $\mu_0(\zeta|_\Gamma)={\deg}_0 A$,\end{enumerate}
and this holds true for the generic choice of coordinates.
\end{prop}
Here $\mu_0(\pi|_A)$ denotes the {\it covering number} of the branched covering $\pi|_A$ with $0$ as the unique point in the fibre $\pi^{-1}(0)$ (see \cite{Ch}) while $m_0(f):=\mu_0(p|_{\Gamma_f})$ is the geometric multiplicity of $f$ at zero.

So as to prove this proposition we begin with a most easy lemma:
\begin{lem}\label{lemmacik}
If $E\subset {\C}^m$ is such that
$\# E=\mu>0$ and $k\leq m$, then for the generic
epimorphism $L\in \mathrm{L}({\C}^m,{\C}^k)$ one has
$\# L(E)=\mu$.
\end{lem}
\begin{proof}
It suffices to prove the assertion for
$k=m-1$. The set $\{\ell\in G_1({\C}^m)\mid \exists x,y\in E\colon
x\not= y, x\in \ell+y\}$ is finite. Thus for the generic $\ell\in
G_1({\C}^m)$ the set $\bigcup_{x\in E} x+\ell$ consists of $\mu$
distinct lines. The orthogonal projection $\pi^\ell$ along $\ell$ is
hence the sought after epimorphism.
\end{proof}

\begin{proof}[Proof of proposition \ref{Rzuty}.]
We know that for the
generic projection $\pi$ we have $\pi^{-1}(0)\cap C_0(A)=\{0\}$ (cf. [Ch]).
Let us take such a projection which in addition `separates' the
points in the maximal fibre of $f$, i.e. for $w$ such that
$f^{-1}(w)$ consists of $m_0(f)$ points, $\pi(f^{-1}(w))$ consists
also of $m_0(f)$ points (cf. the previous lemma). That means $\pi|_{f^{-1}(w)}$ is an injection. 

Observe that $p=\varrho\circ\eta$ and $p|_{\Gamma_f}$,
$\eta|_{\Gamma_f}$, $\varrho|_{\Gamma}$ are proper. Moreover
$m_0(f)=\mu_0(p|_{\Gamma_f})$. 
Now it remains to observe that
$$\varrho^{-1}(w)\cap\Gamma=\pi(f^{-1}(w))\times\{w\}$$
and so by the choice of $\pi$ we have
$\mu_0(\varrho|_\Gamma)=m_0(f)$. Thence 
$\mu_0(\eta|_{\Gamma_f})=1$, which means that
$\eta\colon\Gamma_f\to\Gamma$ is one-to-one. Indeed, if $(x_0,w_0)\in{\C}^k\times{\C}^k$ is fixed, 
\begin{align*}
\#\eta^{-1}(x_0,w_0)\cap \Gamma_f&=\#\{y\in{\C}^{m-k}\mid (x_0,y)\in A,\ f(x_0,y)=w_0\}=\\
&=\#\{z\in f^{-1}(w_0)\mid \pi(z)=x_0\}=\\
&=\#f^{-1}(w_0)\cap \pi^{-1}(x_0).
\end{align*}
The latter is equal to one iff $f|_{\pi^{-1}(x_0)\cap A}$ is injective which is equivalent to $\pi|_{f^{-1}(w_0)}$ being an injection. That we know to be true. Thus, in particular, for any
$x,w\in{\C}^k$ there exists exactly one $y\in{\C}^{m-k}$ such that
$f(x,y)=w$.

Therefore $\mu_0(\zeta|_\Gamma)={\deg}_0 A=:d$. Indeed, if
we take $x_0\in{\C}^k$ near zero such that $\#
\pi^{-1}(x_0)\cap A=d$, then obviously $\#
f(\pi^{-1}(x_0)\cap A)\leq d$. On the other hand if there were
$y\not=y'$ such that $f(x_0,y)=f(x_0,y')=:w_0$, then the set
$\eta^{-1}(x_0,w_0)\cap\Gamma_f$ would include the two points
$(x_0,y,w_0)\not=(x_0,y',w_0)$. If this held true for $x_0$
arbitrarily close to zero, then this would contradict
$\mu_0(\eta|_{\Gamma_f})=1$.
\end{proof}

\smallskip
\noindent{\bf Note.} It may be useful, in reference to \cite{D}, to observe that in the situation from the proposition above it is easy to check that the {\sl \L ojasiewicz exponent} $\mathcal{L}(f;0)={1}/{q_0(\Gamma,\varrho)}$ (with the notations from theorem (2.6) in \cite{D}).

\section{C-holomorphic functions with algebraic graphs}

Let $|\cdot|$ denote any of the usual norms on ${\C}^m$ (we shall not distinguish in notation the norms for different $m$ as long as there is no real need for such a distinction). We begin with the following Liouville-type lemma concerning c-holo\-morph\-ic mappings whose 
graphs are algebraic sets (it is a consequence of the Rudin-Sadullaev criterion):

\begin{lem}\label{2.1} Let $A\subset {\C}^m$ be an analytic set and let 
$f\in{\ch}(A,{\C}^n)$. Then $\Gamma_f$ is algebraic if and only if $A$ is algebraic and there are 
constants $M,s>0$ such that 
$$|f(x)|\leq M(1+|x|^s),\quad \textit{for}\ x\in A.$$
\end{lem}

\begin{proof} 
A simple yet important observation is that given $f\in{\ch}(A,{\C}^n)$, $S\subset \Gamma_f$ is an irreducible component of the graph iff $S=\Gamma_{f|_T}$ and $T\subset A$ is an irreducible component of $A$ (see \cite{Wh}; essentially, it is enough to remark that over a regular point of $A$ we have necessarily a regular point of the graph).

For the `only if' part remark that by Chevalley's Theorem, $A$ is algebraic as the 
proper projection of the graph and so it has finitely many irreducible components. This allows us to assume that $A$ has pure dimension $k\geq 1$. Then, by \cite{L} VII.7.2 we get immediately 
$$\Gamma_f\subset\{(z,w)\in{\C}^m\times{\C}^n\mid |w|\leq M(1+|z|^s)\}$$
for some $M,s>0$.

To prove the `if' part we observe that $\Gamma_f$ has only finitely many irreducible components and again we may assume that $A$ is pure $k$-dimensional (with $k\geq 1$). Then we apply \cite{L} VII.7.4 to $A$ (we may assume now that the considered norms are the $\ell_1$ norms i.e. sum of moduli of the coordinates):
$$A\subset\{(u,v)\in{\C}^k\times{\C}^{m-k}\mid |v|\leq M'(1+|u|)\}$$
(in well-chosen coordinates) for some $M'>0$. Take now $(u,v,w)\in\Gamma_f$, we have then 
$|f(u,v)|\leq M(1+|(u,v)|^s)=M[1+(|u|+|v|)^s]$ and by an easy computation:
\begin{align*}
|f(u,v)|+|v|&\leq M[1+(|u|+|v|)^s]+M'(1+|u|)\leq\\
&\leq 3C(1+|u|)^{s'},
\end{align*}
for $s':=\max\{1,s\}$ and $C:=\max\{M,M(M'+1)^s,M'\}$. Now Rudin-Sadullaev criterion yields $\Gamma_f$ 
algebraic. \end{proof}

\begin{rem}\label{rwz} The condition `$A$ is algebraic' in the 
equivalence is not redundant since any polynomial restricted to e.g. $A=\{y=e^x\}$ satisfies the inequality but has a non algebraic graph (otherwise $A$ would be algebraic too, by the Chevalley-Remmert Theorem).

Note also that $|f(x)|\leq M(1+|x|^s)$ on $A$ iff $|f(x)|\leq M(1+|x|)^s$ on $A$.

More generally, it is a mere exercise to check that if $X\subset {\C}_x^m\times{\C}_y^n$ is a closed set with proper projection $\pi(x,y)=x$, then the following three conditions are equivalent:\begin{enumerate}
\item $\exists R\geq 1, s\geq 0, C>0\colon |y|\leq C|x|^s$, when $(x,y)\in X$ with $|x|\geq R$;
\item $\exists s\geq 0, C>0\colon |y|\leq C(1+|x|^s),\> (x,y)\in X$;
\item $\exists s\geq 0\colon |y|\leq C(1+|x|)^s,\> (x,y)\in X$.
\end{enumerate}
\end{rem}

Hereafter we are interested in particular in {\it c-holomorphic functions with algebraic graphs} which we will call {\it c-algebraic} for short (\footnote{Since `c' in`c-holomorphic' stands for `continuous' we were tempted to propose, in a rather subversive manner, to call such functions {\it al-co-holomorphic functions} (short for `algebraic continuous holomorphic functions'). But one should always resist to temptations.}). We will denote their ring by $\mathcal{O}_c^\mathrm{a}(A)$ when $A\subset{\C}^m$ is a fixed algebraic set. As a matter of fact, we will assume most of the time that $A\subset{\C}^m$ is a pure $k$-dimensional algebraic set of degree $d:=\deg A$ (meaning the degree of the projective completion of $A$). Obviously, we shall assume also $k\geq 1$ unless something else is stated.

For a mapping $f\in\mathcal{O}_c(A,{\C}^k)$, having an algebraic graph is clearly equivalent to each component $f_j$ of $f$ having an algebraic graph. We shall write then $f\in{\oa}(A,{\C}^n)$.

Observe that $P\circ f\in\mathcal{O}_c^\mathrm{a}(A)$, if $f\in\mathcal{O}_c^\mathrm{a}(A)$ and $P$ is a polynomial. More generally, due to the Chevalley-Remmert Theorem, we have an c-algebraic counterpart of Proposition \ref{zlozenie}:
\begin{prop}\label{Zlozenie}
Assume that $X$ and $Y$  are algebraic subsets  of ${\C}^m$ and ${\C}^n$ respectively, and $f\colon X\to Y$ and $g\colon Y\to {\C}$ are c-algebraic mappings. Then $g\circ f\in\mathcal{O}^\mathrm{a}_c(X)$.
\end{prop}
\begin{proof}
The graph of $g\circ f$ is indeed algebraic by the Chevalley-Remmert Theorem as the the proper projection by $\pi(x,y,z)=(x,z)$, where $(x,y,z)\in {\C}^m\times{\C}^n\times{\C}$, of the algebraic set $(\Gamma_f\times{\C})\cap (X\times\Gamma_g)$.
\end{proof}

In view of Lemma \ref{2.1} and in connection with Strzebo\'nski's paper \cite{S}, for any $f\in\mathcal{O}_c^\mathrm{a} (A)$ with any algebraic set $A$, we introduce its {\it growth exponent}
$$
\mathcal{B}(f):=\inf\{s\geq 0\mid |f(x)|\leq C(1+|x|)^s,\ \hbox{\it on}\ A\ \hbox{\it with some constant}\ C>0\}
$$
and the set of all possible growth exponents on $A$:
$$
\mathcal{B}_A:=\{\mathcal{B}(f)\mid f\in\mathcal{O}_c(A)\colon \Gamma_f\ \hbox{\it is algebraic}\}.
$$
Observe that there is in fact 
$$
\mathcal{B}(f)=\inf\{s\geq 0\mid |f(x)|\leq \mathrm{const.}|x|^s,\ x\in A\colon |x|\geq M\ \hbox{\it with some}\ M\geq 1\}.
$$

Recall the following important lemma from \cite{S}:

\begin{lem}[\cite{S} Lemma 2.3]\label{Adam} 
Let $P(x,t)=t^d+a_1(x)t^{d-1}+\ldots+a_d(x)$ be a polynomial with $a_j\in{\C}[x_1,\dots,x_k]$. Then $\delta(P):=\max_{j=1}^d ({\deg a_j}/{j})$ is the minimal exponent $s>0$ for which the inclusion 
$$
P^{-1}(0)\subset \{(x,t)\in{\C}^k\times{\C}\mid |t|\leq C(1+|x|)^s \}
$$
holds with some $C>0$.
\end{lem}

Note that in view of Remark \ref{rwz} it is merely an avatar of the following P\l oski's crucial lemma:
\begin{lem}[\cite{P} lemma (2.1)]\label{P} If $P(x,t)$ is as in the preceding lemma, then $\delta(P)$ is the minimal exponent $q>0$ such that 
$$
\{(x,t)\in{\C}^k\times{\C}\mid P(x,t)=0, |x|\geq R\}\subset\{(x,t)\in{\C}^k\times{\C}\mid |t|\leq C|x|^q\}
$$
for some $R, C>0$. 
\end{lem}
\begin{proof}
For the convenience of the reader we just recall that the proof boils down to observing that any root $t$ of $P(x,t)=0$ with $x$ fixed satisfies $|t|\leq 2\max|a_j(x)|^{1/j}$, and the estimates $|a_j(x)|\leq C_j|x|^{\deg a_j}$ for $|x|\gg 1$.
\end{proof}

From this we easily obtain:
\begin{prop}\label{osiaganie}
Given $f\in{\oa}(A)$, the least upper bound in the definition of $\mathcal{B}(f)$ is attained and a rational number $p/q\geq 0$ with $1\leq q\leq d$ where $d$ is the maximum of the degrees of the components of $A$.
\end{prop}
\begin{proof}
First observe that for the decomposition $A=\bigcup A_j$ into irreducible components, we get $\mathcal{B}(f)=\max \mathcal{B}(f|_{A_j})$. Indeed, since a growth inequality satisfied by $f$ is satisfied by each $f|_{A_j}$ which gives `$\geq$'. On the other hand, if each $f|_{A_j}$ satisfies an inequality at infinity with some $s_j\geq 0$, then for $x\in A$ large enough, $|f(x)|\leq \mathrm{const.}|x|^{\max s_j}$ from which we infer `$\leq$'. 

We may thus assume that $A$ is irreducible. Taking the image (\footnote{It is algebraic since $\pi\times\mathrm{id}_{\C}$ is proper on $\Gamma_f$ due to the continuity of $f$ and the choice of $\pi$.}) $\Gamma\subset{\C}^k\times{\C}$ of the graph $\Gamma_f$ by the projection $\pi\times\mathrm{id}_{{\C}}$, where $\pi\colon{\C}^m\to{\C}^k$ is a projection realizing $\deg A$, it suffices to observe that $\mathcal{B}(f)=\delta(P)$, where $P$ is the minimal polynomial describing the algebraic hypersurface $\Gamma$ and $\delta(P)$ the number from Lemma \ref{Adam}. 

Indeed, by the choice of $\pi$, up to a change of coordinates we may assume that it is the projection onto the first $k$ coordinates, $A\subset \{(x,y)\in{\C}^k\times{\C}^{m-k}\mid |y|\leq \mathrm{const.}(1+|x|)\}$. Therefore, given $(x,t)\in\Gamma$ we find $y\in{\C}^{m-k}$ such that $t=f(x,y)$ and thus, for any growth exponent $q$ of $f$ we get (cf. Remark \ref{rwz})
$$
|f(x,y)|\leq\mathrm{const.}(1+|x|+|y|)^q\leq\mathrm{const.}(1+|x|+(1+|x|))^q.
$$
Hence, $|t|\leq\mathrm{const.}(1+|x|)^q$. On the other hand, $\Gamma\subset\{(x,t)\in{\C}^k\times{\C}\mid |t|\leq\mathrm{const.}(1+|x|)^s\}$ readily implies, that for any $(x,y)\in A$, there is $|f(x,y)|\leq\mathrm{const.}(1+|x|+|y|)^s$, i.e. $s$ is a growth exponent of $f$. Eventually, $\mathcal{B}(f)=\delta(P)$, by Lemma \ref{Adam} (\footnote{Note that $\delta(P)$ does not depend on the choice of $\pi$,, any projection realizing $\deg A$ does the trick.}).

Since $\Gamma$ has proper projection onto ${\C}^k$, $P(x,t)$ written as a polynomial in $t$ with polynomial coefficients has to to be unitary (cf. \cite{TW1}). Its degree in $t$ cannot exceed the cardinality of $f(\pi^{-1}(x))$ for the generic $x$. The result follows.
\end{proof}

The growth exponent replaces in the c-holomorphic setting the notion of the {\it degree} of a polynomial (if $A={\C}^m$, then obviously $\mathcal{O}_c^\mathrm{a}(A)={\C}[x_1,\dots, x_m]$ and so $\mathcal{B}=\mathrm{deg}$ is the usual degree). 

For a  mapping $f=(f_1,\dots, f_n)$ with algebraic graph, defined on $A\subset{\C}^m$, $\mathcal{B}(f):=\max\mathcal{B}(f_j)$ coincides with the least upper bound of exponents $s>0$ for which $|f(x)|\leq\mathrm{const.}|x|^s$, for all $x\in A$ large enough. Similarly, it is easy to check that $\mathcal{B}(f)=\max\mathcal{B}(f|_{A_i})$ where $A=\bigcup_i A_i$ is the decomposition of $A$ into irreducible components.

From Lemma \ref{Adam} via Proposition \ref{osiaganie} we easily obtain also the c-holomorphic counterpart of Strze\-bo\'nski's result from \cite{S}:

\begin{prop}\label{AG2.3} 
We have $$\mathbb{Z}_+\subset\mathcal{B}_A\subset\{p/q\mid p,q\in\mathbb{N}\colon 1\leq q\leq d,\ p,q\ \hbox{\it relatively prime}\},$$
where $d$ is the maximum of degrees of all the irreducible components of $A$.\end{prop}

\begin{rem}\label{zap}
The second inclusion may be strict already when $A$ is an irreducible curve --- see Example \ref{kontr}.
\end{rem}

\smallskip
In the second part of this paper we shall need some more information about $\mathcal{B}(f)$. It is easy to see from the definition that for any $h_1, h_2\in\mathcal{O}_c^\mathrm{a}(A)$ there is $\mathcal{B}(h_1h_2)\leq \mathcal{B}(h_1)+\mathcal{B}(h_2)$ and  $\mathcal{B}(h_1+h_2)\leq\max\{\mathcal{B}(h_1),\mathcal{B}(h_2)\}$. But what will turn out to be most important is that for any positive integer $n$ there is $\mathcal{B}(f^n)=n\mathcal{B}(f)$.

\begin{ex}
Even on an irreducible set $A$ there can be $\mathcal{B}(h_1h_2)< \mathcal{B}(h_1)+\mathcal{B}(h_2)$. Consider the hyperbola $xy=1$ in ${\C}^2$ and $h_1(x,y)=x$, $h_2(x,y)=y$ on it. Then $\mathcal{B}(h_1h_2)=0$, while $\mathcal{B}(h_i)=1$, $i=1,2$.
\end{ex}

\begin{prop}
If $f\in\mathcal{O}_c^{\mathrm{a}}(A)$ is such that $\mathcal{B}(f)=0$, then $f|_{S}\equiv \mathrm{const.}$ for any irreducible component $S\subset A$.
\end{prop}
\begin{proof}
Fix an irreducible component $S\subset A$. Clearly, $\mathcal{B}(f|_S)=0$. Then for a generic proper projection $\pi$ realizing $\deg S$ as it covering number, we have in accordance with Lemma \ref{Adam} and Proposition \ref{osiaganie}, $\mathcal{B}(f|_S)=\delta(P)$ for the minimal polynomial $P$ describing the algebraic hypersurface $(\pi\times \mathrm{id}_{\C})(\Gamma_{f|_S})$. But then $P(x,t)$ has constant coefficients and so we may write $P(x,t)=P(t)=\prod_{j=1}^d (t-t_j)$ and obviously $\{t_1,\dots,t_d\}=f(S)$. By the connectedness of $S$ together with the continuity of $f$, we conclude that $f|_S\equiv\mathrm{const.}$
\end{proof}

It seems interesting to note that no such general Liouville-type result is known for c-holomorphic functions except for the following result based on Cynk's result from \cite{Ck}.
\begin{prop}
Let $A\subset{\C}^k\times{\C}^n$ be an irreducible analytic set of dimension $k$ with proper projection onto ${\C}^k$. Then any bounded c-holomorphic function on $A$ is constant.
\end{prop}
\begin{proof}
Let $f\in\mathcal{O}_c(A)$ be bounded. If $f$ is non-constant, then by \cite{Ck}, $f(A)={\C}\setminus Z$ where $Z$ is finite, which, of course, is impossible when $f$ is bounded.
\end{proof}

\smallskip
Using the results of \cite{TW1} we are able to give an estimate of $\mathcal{B}(f)$. Indeed, by applying \cite{TW1} Theorem 3 we get in Lemma \ref{Adam} above the estimate $\delta(P)\leq\deg P^{-1}(0)-d+1$. It remains to specify what actually $\deg P^{-1}(0)$ and $d$ are.

\smallskip

\begin{thm}\label{oszacowanie}
Let $f\in\mathcal{O}_c^\mathrm{a}(A)$ with $A\subset{\C}^m$ of dimension $k\geq 0$. Then 
$$\mathcal{B}(f)\leq \deg\Gamma_f-\deg A+1.$$
\end{thm}
\begin{proof}
If $f$ is constant or $k=0$, then $\mathcal{B}(f)=0$ and the estimate holds. We may assume thus $f$ non-constant and $k>0$.

Suppose first that $A$ is irreducible. 

Let $\pi\colon {\C}^m\to{\C}^k$ be a projection realizing the degree $\deg A$ and let $\Gamma:=(\pi\times\mathrm{id}_{\C})(\Gamma_f)$. The latter is clearly an algebraic set due to Remmert-Chevalley Theorem. A straightforward application of the main result of \cite{TW1} gives now $\mathcal{B}(f)\leq\mathrm{deg} \Gamma-d+1$, where $d$ is the covering number of the branched covering $\zeta\colon {\C}^k\times{\C}\to{\C}^k$ on $\Gamma$. 

It is easy to see that $\deg\Gamma\leq\deg\Gamma_f$. Indeed, if $\ell\subset{\C}^{k+1}$ is an affine complex line such that $\deg\Gamma=\#(\ell\cap \Gamma)$, then the set $L:=\{z\in{\C}^m\mid (\pi\times\mathrm{id}_{\C})(z)\in \ell\}$ is an affine space of dimension $m+1-k$ intersecting $\Gamma_f$ in a zero-dimensional set. This follows from the properness of $\pi\times\mathrm{id}_{\C}$ on $\Gamma_f$. Therefore, we have $\#(L\cap\Gamma_f)\leq\deg\Gamma_f$ (cf. \cite{L} VII.11). Clearly  $\#(\ell\cap \Gamma)\leq\#(L\cap\Gamma_f)$.

The point is how to choose $\pi$ so as to have $d=\deg A$. 
We are able to do this thanks to Corollary \ref{injekcja} asserting that for the generic $\pi$ we have $\#f(\pi^{-1}(x)\cap A)=\deg A$ for the generic $x\in{\C}^k$.

Now, if $A$ is reducible, then we apply the preceding argument to each irreducible component $S\subset A$ and $f|_S$ getting 
$$
\mathcal{B}(f|_S)\leq\deg \Gamma_{f|_S}-\deg S+1.
$$ Observe that $\mathcal{B}(f)=\max\{\mathcal{B}(f|_S)\mid S\subset A\ \hbox{\it an irreducible component}\}$ (cf. \cite{S}) and $\deg\Gamma_f=\sum_S\deg\Gamma_{f|_S}$, since $\Gamma_{f|_S}$ is irreducible iff $S\subset A$ is irreducible (thus $\Gamma_{f|_S}$ are the irreducible components of the graph). Therefore 
$$
\mathcal{B}(f)\leq\max_{S\subset A}(\deg \Gamma_{f|_S}-\deg S)+1,
$$
but since for each irreducible component $S'$ there is $$\deg \Gamma_{f|_{S'}}-\deg S'\leq\sum_{S\subset A} (\deg \Gamma_{f|_S}-\deg S),$$ we finally obtain the required inequality.
\end{proof}
The following example shows that the estimate is far from being the best one. Nevertheless, it is of some interest in view of the proposition following the example.
\begin{ex}\label{przyklad}  Let $A$ be the algebraic curve 
$\{(x,y)\in {\C}^2\mid x^3=y^2\}$ and consider the c-holomorphic function
$$
f(x,y)=
\begin{cases}
\displaystyle
\frac{y}{x},\> \textit{\ for}\> (x,y)\in A\setminus (0,0)\\
0, \>\textit{\ for}\> x=y=0
\end{cases}
$$
It is easy to see (cf. Lemma \ref{2.1}) that $\Gamma_f$ is algebraic and $\mathcal{B}(f)=1/3$. Actually the algebricity of $\Gamma_f$ is not really 
surprising because $f$ is the restriction to $\mathrm{Reg} A$ of a rational function. We will 
show (next section) that in fact the algebricity of the graph is equivalent in this case to the fact that 
the function has a rational `extension'.

Let us observe that the restriction of a polynomial to $A$ may have a growth exponent  $<1$. Consider on the same set $A$ the function $g(x,y)=x$. A straighforward computation yields $\mathcal{B}(g)=2/3$ (see also Theorem \ref{krzywa}).

Finally, let us note that a c-algebraic function is to some,but rather limited, extent `represented' by the projection onto the target space when restricted to its graph. Such a restriction $\pi|_{\Gamma_f}$ is a regular function but unfortunately it does not encode all the information about $f$. We can see the reason in the simplest case of a polynomial: if we restrict the projection $\pi(x,y)=y$ to the graph of the function $p(x)=x^2$, we get $\mathcal{B}(\pi|_{\Gamma_p})=1<\mathcal{B}(p)=\deg p=2$ and there is no simple relation between the two exponents in general.
\end{ex}
\begin{prop}\label{rwn}
 Let $f\in\mathcal{O}^\mathrm{a}_c(A)$ with $A$ pure $k$-dimensional. Then the following conditions are equivalent:
\begin{enumerate}
 \item $\mathcal{B}(f)\leq 1$;
\item $\deg A=\deg\Gamma_f$;
\item $\Gamma_f^*\cap (\{0\}\times{\C})=\{0\}$;
\item $\{P=0\}^*\cap (\{0\}\times{\C})=\{0\}$,
\end{enumerate}
where $P$ is the minimal polynomial describing $(\pi\times\mathrm{id}_{\C})(\Gamma_f)$ with $\pi\colon {\C}^m\to {\C}^k$ a projection realizing $\deg A$. 
\end{prop}
\begin{proof}
%
The implication from $(2)$ to $(1)$ follows from the preceding theorem. To prove the converse, suppose that coordinates in ${\C}^m={\C}^k_x\times{\C}_y^{m-k}$ are chosen in such a way that the projection onto the first $k$ coordinates realizes $\deg A$. Then $A\subset\{|y|\leq M(1+|x|)\}$. On the other hand, since $\mathcal{B}(f)\leq 1$, there is $|f(x,y)|\leq C(1+|x|+|y|)$ for $(x,y)\in A$. Combining the two facts, we obtain
$$
\Gamma\subset\{(x,y,t)\in{\C}^k\times{\C}^{m-k}\times{\C}\mid |y|+|t|\leq [M+C(1+M)](1+|x|)\}
$$
which implies that the projection $(x,y,t)\mapsto x$ realizes $\deg\Gamma_f$. Clearly, by the univalence of the graph, the multiplicity of this projection is equal to $\deg A$. 

Now, we turn to proving the equivalence of the remaining conditions. 
Choose affine coordinates so as to have ${\C}^m_z={\C}^k_x\times{\C}^{m-k}_y$ and $\pi(z)=x$. Write $P(x,t)=t^d+a_1(x)t^{d-1}+\ldots+a_d(x)$ and observe that $\mathrm{deg} P=\max\{d,\mathrm{deg} a_1+d-1,\dots, \mathrm{deg} a_d\}$. Therefore, in view of the fact that $\mathcal{B}(f)=\delta(P)$, $(1)$ is equivalent to $\mathrm{deg} P=d$. This in turn is equivalent to $(4)$, since $\{P=0\}^*=\{P^+=0\}$. 

Now, clearly $\Gamma_f^*\supset \{0\}\times{\C}$ implies $\{P=0\}^*\supset\{0\}\times{\C}$ and thus $(4)$ implies $(3)$. On the other hand, by the choice of $\pi$, 
$A\subset\{|y|\leq c(1+|x|)\}$ and so for all $(x,y)\in A$ large enough, $|y|\leq c'|x|$. Take a point $(0,t)\in\{P=0\}^*$. There are sequences $\{P=0\}\ni (x_\nu, t_\nu)\to\infty$ and $\lambda_\nu\in{\C}$ such that $\lambda_\nu(x_\nu, t_\nu)\to (0,t)$. Of course, there is a sequence $y_\nu\in{\C}^{m-k}$ such that $(x_\nu,y_\nu)\in A$ and $t_\nu=f(x_\nu,y_\nu)$. For all $\nu$ large enough, $|\lambda_\nu y_\nu|\leq c'|\lambda_\nu x_\nu|\to 0$, whence $(0,t)\in \Gamma_f^*$. Therefore, $(3)$ implies $(4)$.
\end{proof}

\section{Algebraic Graph Theorem}

Using Oka's theorem about universal denominators (cf. \cite{TsY} and \cite{Wh}) one can show that any c-holomorphic 
function admits locally a universal denominator. We will detail this a little more  
in the proof of the following theorem. For the convienience of the reader let us start with one useful construction of a universal denominator.
\begin{prop}\label{3.3} Let $A\subset U\times{\C}_t\times{\C}_y^{m-k}$ be a pure $k$-dimensional analytic set, where $U\subset{\C}_x^{k}$ is open and connected, such that $0\in A$ and the natural projection $\pi(x,t,y)=x$ is proper on $A$ with covering number $d$. Then after a change of coordinates in ${\C}\times{\C}^{m-k}$ there exists a monic polynomial $P\in\mathcal{O}(U)[t]$ of degree $d$ such that $Q(x,t,y):=\frac{\partial P}{\partial t}(x,t)$ is a universal denominator at each point $a\in A$. 
\end{prop}

\begin{proof} Let $\rho(x,t,y)=(x,t)$ and $\xi(x,t)=x$ be the natural projections. For any point $x\in U$ not critical for $\pi|_A$ we have exactly $d$ distinct points $(t_1,y^1),\dots, (t_d, y^d)$ over it in $A$. If we fix $x$, then taking if necessary a rotation in ${\C}\times{\C}^{m-k}$, we may assume that all the points $t_1, \dots, t_d$ are distinct. Thus $\xi$ on $\rho(A)$ has multiplicity $d$ as a branched covering. Note that by the Remmert theorem $\rho(A)\subset U\times{\C}$ is an analytic hypersurface. Thus there exist a reduced Weierstrass polynomial $P\in\mathcal{O}(U)[t]$ such that $P^{-1}(0)=\rho(A)$. Its degree is obviously $d$. 

Now for fixed $x$ in a simply connected neighbourhood $V$ not intersecting the critical set of $P$ we have $(t_1(x),y^1(x)),\dots, (t_d(x),y^d(x))$, exactly $d$ distinct points. Given a weakly holomorphic function (\footnote{I.e. a function defined on the regular part $\mathrm{Reg} A$, holomorphic there and locally bounded near the singularities $\mathrm{Sng}A$. Then normality is just the property that each such a function is locally the restriction of a holomorphic function in the ambient space.}) $f\colon \mathrm{Reg} A\to {\C}$ put
$$
h(x,t):=\sum_{j=1}^d f(x,t_j(x),y^j(x))\prod_{\iota\neq j} (t-t_\iota(x)),\quad (x,t)\in V\times {\C}.
$$
Observe that $h(x,t_j(x))=f(x,t_j(x),y^j(x))Q(x,t_j(x))$. Clearly the function $h(x,t)$ is locally holomorphic apart from the critical set of $f$ (because the functions $t_j(x)$ and $y^j(x)$ are locally holomorphic) and locally bounded near the critical points of $P$, and so by the Riemann theorem we obtain a holomorphic function $h\in\mathcal{O}(U\times {\C}\times{\C}^{m-k})$ being an extension of $Qf$. \end{proof}

\begin{thm}\label{3.4} Let $A\subset{\C}^m$ be a purely $k$-dimensional algebraic 
set and let $f\in{\ch}(A)$. Then $\Gamma_f$ is algebraic if and only if there exists a rational 
function $R\in{\C}(x_1,\dots,x_m)$ equal to $f$ on $A$ (in particular $R|_A$ is continuous). More precisely, there exists a polynomial $Q\in{\C}[x_1,\dots,x_m]$ with $\mathrm{deg} Q<\mathrm{deg} A$ such that $f=P/Q$ on $A$ for some polynomial $P\in{\C}[x_1,\dots, x_m]$.\end{thm}

\begin{proof} If $m=1$, then either $A$ is the whole ${\C}$, and then by the identity principle 
$\Gamma_f$ is algebraic if and only if $f$ is a polynomial (cf. Serre's theorem on the algebraic graph), 
or $A$ is a finite set and we apply a Lagrange interpolation. In both cases $Q\equiv 1$. Hence we may confine us to the case 
$m\geq 2$.

If $k=0$, then $\#A<\infty$. 
We follow Lemma \ref{lemmacik}. The set $$\{\ell\in G_1({\C}^m)\mid \exists x,y\in A,\ x\neq y\colon y\in x+\ell\}$$ is finite (even algebraic). Take thus a line $\ell\in G_1({\C}^m)$ such that for all $x\in A$, $(x+\ell)\cap A=\{x\}$. If we denote by $\pi^\ell$ the natural projection along $\ell$ onto its orthogonal complement $\ell^\bot$, then $\#\pi^\ell(A)=\# A$. Continuing this procedure we find a one-dimensional subspace $L\subset{\C}^m$ such that $\#\pi(A)=\#A$, where $\pi$ is the orthogonal projection onto $L$. Now the Lagrange interpolation for $\pi(A)$ and the values $f(a), \pi(a)\in\pi(A)$ yields a polynomial $P\in{\C}[t]$. Then $\widetilde{P}(x):=P(\pi(x))$ is the polynomial interpolating $f$ on $A$. The `only if' part is clear.

\smallskip
Assume now that $k\geq 1$. 
Since $A$ is 
algebraic of pure dimension $k$, there are coordinates in ${\C}^m$ such that the projection 
$\pi$ onto the first $k$ coordinates is proper on $A$ (so it is a branched covering) and it realizes $\mathrm{deg} A$. 
%
Now if we take $\rho(x_1,\dots, x_m)=(x_1,\dots, x_{k+1})$, then we are able to apply proposition \ref{3.3} getting a polynomial (since by Chevalley's theorem $\rho(A)$ is an algebraic hypersurface) $Q$ being a local universal denominator for $A$. Since ${\C}^m$ is a domain of holomorphy, $Q$ is in fact a global universal denominator for $A$ (actually this follows directly from the proof of \ref{3.3}).

That means that there exists $h\in\mathcal{O}({\C}^m)$ such that 
$$
f=\frac{h}{Q}\quad \hbox{on}\ \mathrm{Reg} A.
$$
Note that for points $a\in\mathrm{Sng} A$, if we take any sequence $\mathrm{Reg} A\ni a_\nu\to a$, then by continuity we obtain $f(a)Q(a)=h(a)$. Therefore either $a$ is a point in which $h/Q$ is well defined, or it is a point of indeterminacy of $h/Q$. In the latter case, the function $h/Q$ has a finite and well defined limit along $A$, namely $f(a)$. Thus $h/Q$ is continous on $A$.  

\smallskip
As a matter of fact $h$ is not uniquely determined. The proof of our theorem consists now in showing 
\begin{enumerate}
\item[(a)] If $h$ is a polynomial, then $\Gamma_f$ is algebraic;

\item[(b)] If $\Gamma_f$ is algebraic, then we may choose $h\in{\C}[x_1,\dots,x_m]$.
\end{enumerate}
\noindent Ad (a): 
%
Let $X:=(A\times{\C})\cap\{(x,t)\in{\C}^m\times{\C}\mid h(x)=Q(x)t\}$. It is an algebraic set of dimension at least $k$. Over points $x\in A\setminus Q^{-1}(0)$ this is exactly the graph of $f$. Thus for each such point $x$ and the only one  $t$ for which $(x,t)\in X$, we have $\mathrm{dim}_{(x,t)} X=k$. On the other hand, since $Q$ does not vanish on any irreducible component of $A$, the set $A\cap Q^{-1}(0)$ has pure dimension $k-1$ (see \cite{D2}). For each point $x\in A\cap Q^{-1}(0)$ we have a whole line $\{x\}\times{\C}\subset X$. Thus the set $X$ has pure dimension $k$. 

Set $\Gamma:=\Gamma_f\setminus(Q^{-1}(0)\times{\C})=\Gamma_f\setminus [(A\cap Q^{-1}(0))\times{\C}]$. Then we have $\Gamma\subset X$ and so for closures $\overline{\Gamma}\subset\overline{X}=X$. But by continuity $\overline{\Gamma}=\Gamma_f$, and since $\Gamma_f$ has pure dimension $k$ it must be the union of some irreducible components of $X$. Since $X$ is algebraic, so is $\Gamma_f$. 

\smallskip
\noindent Ad (b): This follows from Serre's algebraic graph theorem (for regular functions, see \cite{L}). Indeed, $fQ$ is a holomorphic function in ${\C}^m$ with algebraic graph over the algebraic set $A$ (to see this apply lemma \ref{2.1}; one can remark by the way that $\mathcal{B}(fQ)\leq\mathcal{B}(f)+\mathcal{B}(Q|_A)$). Thus it is on $A$ a regular function which means that it is in fact the restriction to $A$ of a polynomial $P$. \end{proof}

\begin{rem}
It is easy to check that in the theorem above we obtain $$\mathcal{B}(f)\geq\mathcal{B}(P|_A)-\mathcal{B}(Q|_A).$$ 
\end{rem}



\section{Generically finite c-holomorphic mappings with algebraic graphs}\label{proper}
C-holomorphic functions with algebraic graphs are a promising generalization of polynomials onto algebraic sets. Most of the theorems known for instance for polynomial dominating mappings should have their analogues at least for c-holomorphic proper mappings with algebraic graphs. Note, however, that in this setting we are naturally obliged to make do more with the geometric structure than the algebraic one (that is a hindrance when trying to extend the results of \cite{D2} to the c-holomorphic algebraic case).

We consider now the following situation:\\ Let $A\subset{\C}^m$ be algebraic of pure dimension $k>0$ and $f\in\mathcal{O}_c^\mathrm{a}(A,{\C}^k)$. Suppose first $f$ is a proper mapping. 

Since $\Gamma_f$ is algebraic with proper projection onto ${\C}^k$, then $\# f^{-1}(w)$ is constant for the generic $w\in{\C}^k$. We call this number, denoted by $\mathrm{d}(f)$, the {\it geometric degree} of $f$ just as in the polynomial case. We call {\it critical} for $f$ any point $w\in{\C}^k$ for which $\#f^{-1}(w)\neq\mathrm{d}(f)$. In that case one has actually $\#f^{-1}(w)<\mathrm{d}(f)$ (cf. e.g. \cite{Ch}, the projection onto ${\C}^k$ restricted to $\Gamma_f$ is a $\mathrm{d}(f)$-sheeted branched covering). Obviously $\mathrm{d}(f)\leq\mathrm{deg}\Gamma_f$ (cf. \cite{L}).

Of course, we could define this degree in a more general setting. Observe that the properness of a continuous mapping $f\colon A\to{\C}^n$ is equivalent to $\lim_{x\in A\colon |x|\to+\infty}|f(x)|=+\infty$. Of course, if the graph of $f$  is algebraic, then the properness of $f$ implies $n\geq k$. Now, following Z. Jelonek, we say that $f$ is {\it proper over} $y\in{\C}^n$, if $y$ admits a compact neighbourhood $K$ such that $f^{-1}(K)$ is compact. Assume $f$ is c-algebraic. If we denote by $J_f$  the set of points over which $f$ is not proper (the {\it Jelonek set}, see e.g. \cite{Jel1} or \cite{Jel2}), then it coincides with $\rho(\overline{\Gamma_f}\setminus\Gamma_f)$ where the closure is taken in $\mathbb{P}_m\times{\C}^m$ (here $\mathbb{P}_m=\overline{{\C}^m}$) and $\rho\colon\mathbb{P}_m\times{\C}^n\to{\C}^n$ is the natural projection. Therefore, by the Chevalley-Remmert Theorem it is a constructible set, hence an algebraic one due to its closedness. Moreover, $\dim J_f<k=\dim\Gamma_f$. Now, if $n=k$ and $f$ is dominant, i.e. $\overline{f(A)}={\C}^k$, its fibres must be generically finite (see the next Lemma), since otherwise we would get $\dim A>k$.  Clearly, $f$ is a branched covering over the connected manifold ${\C}^k\setminus J_f$ and so the generic cardinality of the fibre $\mathrm{d}(f)$ is well-defined. Actually, it can also be easily defined that way for the case $n>k$, provided $A$ is irreducible, since $f$ is then a branched covering over the $k$-dimensional connected manifold $\mathrm{Reg}\overline{f(A)}\setminus J_f$. 

Another way of defining $\mathrm{d}(f)$ would be just to take directly, just as in \cite{Jel2}, the generic multiplicity of the regular mapping $\pi|_{\Gamma_f}$ where $\pi$ is the projection onto the target space. It is a classical fact that $J_f$ is characterized by the alternative: $\dim\pi^{-1}(y)\cap \Gamma_f>0$ or the fibre is finite but the sum of the local multiplicities is $<\mathrm{d}(f)$. It may be interesting to note the following easy Lemma.
\begin{lem}\label{S}
Let $X\subset{\C}_x^n\times{\C}_y^k$ be algebraic of pure dimension $k$ and $p(x,y)=y$ denotes the natural projection. Put $$S(p|_X):=\{y\in{\C}^k\mid \dim p^{-1}(0)\cap X=0\}.$$
Then the following asssertions are equivalent:\begin{enumerate}
\item $S(p|_X)\neq\varnothing$;
\item $\operatorname{int}S(p|_X)\neq\varnothing$;
\item $\overline{S(p|_X)}={\C}^k$.
\end{enumerate}
\end{lem}
\begin{proof}
(1) implies (2) since for any $y\in S(p|_X)$, there are two neighbourhoods $V\ni y$ and $U\supset p^{-1}(y)$ such that $(\partial U\times V)\cap X=\varnothing$ and so $p|_{(U\times V)\cap X}$ is proper which means by the Remmert Theorem that $V\subset S(p|_X)$. 

(2) implies (3), for $S(p|_X)$ is a constructible set of dimension $k$. Finally, (3) implies (1), obviously.
\end{proof}
Clearly, for a dominant $f\in{\oa}(A,{\C}^k)$, ${\C}^k\setminus S(p|_{\Gamma_f})\subset J_f$. 

Finally, observe that if $A=\bigcup_{j=1}^r A_j$ is the decomposition into irreducible components, then for each $A_j$, by the preceding Lemma, either $f|_{A_j}$  is dominant (which is true already when $f|_{A_j}$ has a finite fibre) and then $\mathrm{d}(f|_{A_j})$ is well-defined, or $\dim\overline{f(A_j)}<k$ in which case no fibre of $f|_{A_j}$ is finite (compare \cite{L} V.3.2 Theorem 2). Therefore, in the second case $\mathrm{d}(f|_{A_j})$ is not defined and we may consider $f(A_j)$ as {\it irrelevant}. Now, if $f|_{A_j}$ is dominant for $j=1,\dots, s$ and not dominant for $j=s+1,\dots, r$, then we may define $\mathrm{d}(f):=\sum_{j=1}^s\mathrm{d}(f|_{A_j})$ and we obviously have $\mathrm{d}(f)=\mathrm{d}(f|_{A_1\cup\ldots\cup A_s})$.

%


Before proving a B\'ezout inequality, let us note here also that a c-holomorphic proper mapping satisfies {\it the \L ojasiewicz inequality at infinity}:
\begin{thm}
Let $f\in{\oa}(A,{\C}^n)$ be proper with $A\subset{\C}^m$ pure $k$-dimensional, $k\geq 1$. Then there are constants $R, C,\ell>0$ such that 
$$
|f(x)|\geq C|x|^\ell,\quad |x|\geq R.
$$
\end{thm}
\begin{proof}
By assumptions, the natural projection $p(x,y)=y$ from ${\C}^m\times{\C}^n$ to ${\C}^n$ is proper on $\Gamma_f$. 
Necessarily, $n\geq k$ and $X:=f(A)$ is algebraic pure $k$-dimensional by the Chevalley Theorem. For the generic 
choice of coordinates in ${\C}_y^n={\C}_u^k\times{\C}_v^{n-k}$, the natural projection $\pi(u,v)=u$ is proper on $X$. Write $\eta(u,v)=v$ and assume that the norm on ${\C}^n$ is chosen in such a way that $|y|=|u|+|v|$. 

Now, $\pi\circ p$ is proper on $\Gamma_f$ and thus by \cite{TW1}, there are costants $c,q>0$ such that
$$
\Gamma_f\subset\{(x,u,v)\in{\C}^m\times{\C}^k\times{\C}^{n-k}\mid |x|+|v|\leq c(1+|u|)^q\}.
$$
From this we get for $x\in A$, and in view of the fact that $q>0$,
\begin{align*}
|x|&\leq |x|+|\eta(f(x))|\leq c(1+|\pi(f(x))|)^q\leq\\ &\leq c(1+|\pi(f(x))|+|\eta(f(x))|)^q=c(1+|f(x)|)^q
\end{align*}
which for some $R\geq 1$ and $c'>0$ is equivalent to 
$$
|x|\leq c'|f(x)|^q,\quad x\in A, |x|\geq R,
$$
by Remark \ref{rwz}.
\end{proof}
Since the inequality from the Theorem above is satisfied with any exponent $\ell'\leq\ell$, it is natural to introduce the {\it \L ojasiewicz exponent at infinity} defined as
$$
\mathcal{L}_\infty(f):=\sup\{\ell>0\mid |f(x)|\geq\mathrm{const.}|x|^\ell, x\in A,|x|\gg 1\}.
$$
A more detailed study of this exponent will be presented in \cite{BDT}. Here we give only two results, Theorems \ref{krzywa} and \ref{AB} below.

Similarly to the polynomial case, we have the following B\'ezout-type theorem (compare also Proposition 4.6 in \cite{D}):
\begin{thm}\label{Bezout}
Let $f\colon A\to{\C}^k$ be a c-holomorphic dominant mapping with algebraic graph. Then 
$$
\mathrm{d}(f)\leq\mathrm{deg} A\prod_{j=1}^k\mathcal{B}(f_j).
$$
Moreover, $\deg A$ can be replaced above by the degree of 
$$
A'=\bigcup\{S\subset A\mid S\>\textit{an irreducible component}\colon \overline{f(S)}={\C}^k\}.
$$
\end{thm}
\begin{proof}
Let $q_j$ be any positive integers such that $q_j\mathcal{B}(f_j)\in \mathbb{N}$ for $j=1,\dots,k$. Then set $F:=(f_1^{q_1},\dots,f_k^{q_k})$. We still have $F\in\mathcal{O}_c^\mathrm{a}(A)$ and $F$ is 
dominant with $\mathrm{d}(F)=\mathrm{d}(f)\prod_j q_j$. Besides, $\mathcal{B}(F_j)=q_j\mathcal{B}(f_j)$. 

The idea now is to follow the idea used in the proof of proposition (4.6) from \cite{D} inspired by the methods of P\l oski and Tworzewski. To that aim consider the algebraic set 
$$
\Gamma:=\{(z,w)\in A\times{\C}^k\mid w_j^{\mathcal{B}(F_j)}=F_j(z),\ j=1,\dots, k\}.
$$
Clearly, for any $a\in \Gamma$, there is $\mathrm{dim}_a\Gamma\geq k$ and since $\Gamma$ has proper projection $p(z,w)=z$ onto $A$, the converse inequality holds too and so $\Gamma$ is pure $k$-dimensional. 

Assume for the moment that $A$ is irreducible. 

Take now any affine subspace $\ell\subset{\C}^m$ of dimension $k$ such that $\#(\ell\cap A)=\mathrm{deg} A$ and 
$$
A\subset\{x+y\in\ell^\bot+\ell\mid |y|\leq C(1+|x|)\},
$$
where $\ell^\bot$ is an orthogonal complementary to $\ell$, $x+y=z$ and $C>0$ a constant. Then by construction $L:=\ell+{\C}^k$ (seen in ${\C}^{m+k}$) is transversal to $\Gamma$ and we have $\#(L\cap \Gamma)=\mathrm{deg} A\prod_j\mathcal{B}(F_j)$. Indeed, since $F$ is dominant, none of the $F_j$'s can be constant and thus, by the identity principle (see \cite{D2}) $F_j^{-1}(0)$ has pure dimension $k-1$. This means that for the generic $x\in\ell^\bot$ and any $z=x+y\in A$ the equation $w_j^{\mathcal{B}(F_j)}=F_j(z)$ has $\mathcal{B}(F_j)$ distinct solutions $w_j$.

We may assume that the norm in consideration is the sum of moduli. Now observe that for $(z,w)\in\Gamma$,
$$
|w_j|^{\mathcal{B}(F_j)}=|F_j(z)|\leq c_j|z|^{\mathcal{B}(F_j)}\ \hbox{when}\ |z|\geq R_j,
$$
for some $c_j,R_j>0$. Then $|w|\leq(\max_j c_j)|z|$ when $|z|\geq \max_j R_j$. Therefore, there exists a constant $K>0$ such that 
$$
\Gamma\subset\{(x,y,w)\in\ell^\bot+\ell+{\C}^k\mid |y|+|w|\leq K(1+|x|)\}
$$
and so $\mathrm{deg}\Gamma=\mathrm{deg} A\prod_j\mathcal{B}(F_j)$. 

Finally, it suffices to remark that one has $\mathrm{d}(F)\leq\mathrm{deg}\Gamma$ since there is  $\mathrm{d}(F)=\#(({\C}^m\times\{w\}^k)\cap\Gamma)$ for a well chosen $w$. 

Once we have the theorem for an irreducible $A$, in the general case we sum the obtained inequalities  for each irreducible component $A_\subset A$ on which $F$ is dominant (these are exactly the components on which $f$ is dominant):
\begin{align*}
\mathrm{d}(f|_{A_j})\prod_i q_i=\mathrm{d}(F|_{A_j})&\leq \deg A_j\prod_{i} \mathcal{B}(F_i|_{A_j})\leq \\
&\leq \deg A_j\prod_{i} \mathcal{B}(F_i)=\deg A_j\prod_{i} \mathcal{B}(f_i)q_i,
\end{align*}
since $\mathcal{B}(F_i)=\max_j\mathcal{B}(F_i|_{A_j})$ (these are natural numbers). Summing up,
$$
\sum_j\mathrm{d}(f|_{A_j})=\mathrm{d}(f)\leq \deg A'\prod_i\mathcal{B}(f_i)\leq \deg A\prod_i\mathcal{B}(f_i),
$$
as required.
\end{proof}

\begin{ex}
Let $A$ and $f$ be as in Example \ref{przyklad}. 
Since $f$ is injective, $\mathrm{d}(f)=1$. Clearly $\mathrm{deg} A=3$ and $\mathcal{B}(f)=1/3$. Thus $\mathrm{d}(f)=\mathrm{deg} A\cdot\mathcal{B}(f)$.
\end{ex}

This example hints at a more general observation concerning images of polynomial generic injections. Let us recall the following simple lemma.
\begin{lem}
Let $\gamma\colon {\C}\to {\C}^m$ be a polynomial generic injection i.e. a polynomial mapping for which there is a finite set $Z\subset\gamma({\C})$ such that $\gamma|_{{\C}\setminus \gamma^{-1}(Z)}\colon {\C}\setminus \gamma^{-1}(Z)\to \gamma({\C})\setminus Z$ is injective. Then $\gamma$ is proper and $\gamma({\C})$ is an irreducible curve of degree $\deg\gamma:=\max\deg\gamma_j$.
\end{lem}
\begin{proof}
The properness follows from the non-constancy and then the Remmert-Chevalley Theorem ensures that $\Gamma:=\gamma({\C})$ is an algebraic curve which obviously has to be irreducible. As for the degree, take an affine hyperplane $H\subset{\C}^m$ such that $H\cap \Gamma=H\cap \Gamma\setminus Z$ and $\#(H\cap \Gamma)=\deg \Gamma$. Then $H=L^{-1}(c)$ for some linear form $L$ and $c\in{\C}$. Then $L\circ\gamma$ is a polynomial of degree $\leq \deg\gamma$ and $\deg\Gamma=\#(L\circ\gamma)^{-1}(c)\leq\deg\gamma$. It remains to show that this inequality is not strict.

Let us assume that $d:=\deg\gamma=\deg\gamma_1=\ldots=\deg\gamma_n>\deg\gamma_{n+1}\geq\ldots\geq\deg\gamma_m$. Write $L(x)=\sum_{j=1}^ma_jx_j$ and $\gamma_j(t)=\sum_{i=1}^{d_j}c_{ji}t^i$ with $d_j=\deg\gamma_j$. The inequality obtained so far is strict only when $\sum_{j=1}^na_jc_{jd}=0$ i.e. $L(c_{1d},\dots, c_{nd},0,\dots,0)=0$. Let $\ell$ be the line spanned by the vector $(c_{1d},\dots, c_{nd},0,\dots,0)$. It is easy to check that $\ell=\Gamma^*$ and since by the choice of $H$, we have $H^*\cap \Gamma^*=\{0\}$ and $H^*=\Ker L$, then we are done.
\end{proof}
\begin{thm}\label{krzywa}
Let $\gamma\colon {\C}\to {\C}^m$ be a polynomial generic injection and $f\colon \Gamma\to{\C}$ a non-constant c-algebraic function on the algebraic curve $\Gamma:=\gamma({\C})$. 
Then
$$
\mathcal{L}_\infty(f)=\mathcal{B}(f)=\frac{\mathrm{d}(f)}{\mathrm{deg} \Gamma}=\frac{\deg(f\circ\gamma)}{\deg\gamma}.
$$
\end{thm}
\begin{proof}
The idea of the proof is similar to that of theorem (3.2) from \cite{D}. We may assume that $|\cdot|$ is the maximum norm. Let $d:=\mathrm{deg}\Gamma$ and observe that since byt the preceding Lemma,
$$
\lim_{|t|\to+\infty} \frac{|\gamma(t)|}{|t|^d}=\mathrm{const.}>0,
$$
we have $c|t|^d\leq |\gamma(t)|\leq C|t|^d$ for $|t|\gg 1$ and some constants $c,C>0$. Then 
the inequalities $c_1|x|^\ell\leq|f(x)|\leq c_2 |x|^{b}$ for $x\in \Gamma$ with $|x|\gg 1$ and exponents $\ell,b\geq 0$ and constants $c_1, c_2>0$ are equivalent to 
$$
c_1|\gamma(t)|^{d\ell}\leq|f(\gamma(t))|\leq c_2\cdot |\gamma(t)|^{db},\quad |t|\gg 1.
$$
Note that we clearly have $\ell\leq b$, i. e. $\mathcal{L}_\infty(f)\leq \mathcal{B}(f)$.

Observe now that $f\circ \gamma$ is a polynomial by Serre's Graph Theorem and so there are two positive constants $c'_1,c'_2$ such that
$$
c'_1 |t|^{\mathrm{deg} (f\circ\gamma)}\leq |f(\gamma(t))|\leq c'_2\cdot |t|^{\mathrm{deg} (f\circ\gamma)},\quad |t|\gg 1.\leqno{(\circ)}
$$
But ${\mathrm{deg} (f\circ\gamma)}=\mathrm{d}(f)$ because $\gamma$ being generically injective, we obviously have $\#(f\circ\gamma)^{-1}(w)=\#f^{-1}(w)$ for the generic $w\in{\C}$. Therefore, $\mathrm{d}(f)\leq db$ and so $\mathrm{d}(f)\leq \deg\Gamma\cdot\mathcal{B}(f)$. 

We have now on the one hand for $|t|\gg 1$,
$$
|f(\gamma(t))|\leq c_2'|t|^{\mathrm{d}(f)}\leq c_2'c|\gamma(t)|^{\mathrm{d}(f)/d}
$$
which implies $\mathcal{B}(f)\leq \mathrm{d}(f)/d$. Eventually, $\mathcal{B}(f)\cdot \deg\Gamma=\mathrm{d}(f)$. 

On the other hand, 
$$
|f(\gamma(t))|\geq c_1'|t|^{\mathrm{d}(f)}\geq c_1'c|\gamma(t)|^{\mathrm{d}(f)/d}
$$
implies readily $\mathcal{L}_\infty(f)\geq \mathrm{d}(f)/d$. But we already know that $\mathrm{d}(f)/d=\mathcal{B}(f)$ and thus $$\mathcal{L}_\infty(f)=\mathcal{B}(f)=\frac{\mathrm{d}(f)}{\deg \Gamma}.$$
\end{proof}
\begin{rem}
The statement of the Theorem above remains true for $f\in{\oa}(\Gamma,{\C}^n)$ with $n>1$ as can be easily seen from the proof. In that case $\mathrm{d}(f)$ is defined in a slightly different way. Namely, 
$$
\mathrm{d}(f):=\max_{j=1}^n\mathrm{deg} (f_j\circ\gamma),
$$
i.e. $\mathrm{d}(f)=\mathrm{deg} (f\circ\gamma)$ ($f\circ\gamma$ is a polynomial mapping). 
\end{rem}
Now we can give the Example announced in Remark \ref{zap}.
\begin{ex}\label{kontr}
Consider once again $A\colon y^2=x^3$ of projective degree 3. According to the last Theorem, for any non-constant $f\in{\oa}(A)$, we have $\mathcal{B}(f)=\mathrm{d}(f)/3$. Since $\mathrm{d}(f)$ coincides with the degree of the polynomial $f\circ \gamma$ where $\gamma(t)=(t^2,t^3)$, we see that we it is impossible to obtain e.g. $\mathcal{B}(f)=1/2$. Therefore, the second inclusion in Proposition \ref{AG2.3} is in this case strict.
\end{ex}
Using the methods from the proof of Theorem \ref{Bezout}, we obtain an upper bound for the \L ojasiewicz exponent at infinity:
\begin{thm}\label{AB}
Let $f\in{\oa}(A,{\C}^k)$ be proper with $A\subset{\C}^m$ pure $k$-dimensional, $k\geq 1$. Then 
$$
\mathcal{L}_\infty(f)\leq \sqrt[k]{\frac{\mathrm{d}(f)}{\deg A}}.
$$
\end{thm}
\begin{proof}
Take an exponent $\ell>0$ such that $|f(x)|\geq \mathrm{const.}|x|^\ell$ for $x\in A$, $|x|\gg 1$. Fix an integer $q>0$ such that $q\ell\in\mathbb{N}$ ad put $F_j:=f_j^q$. The resulting mapping $F=(F_1,\dots, F_k)$ is c-algebraic, proper and clearly $\mathrm{d}(F)=q^k\mathrm{d}(f)$. Moreover, taking e.g. the maximum norm we see that $|F(x)|=|f(x)|^q\geq \mathrm{const.}|x|^{q\ell}$ for $x\in A$ sufficiently large. Put $L:=q\ell$ and let
$$\Gamma=\{(x,y)\in A\times{\C}^k\mid y_j^L=F_j(x),\ j=1,\dots, k\}.$$
It is an algebraic set of pure dimension $k$ (cf. the proof of Theorem \ref{Bezout}). Choose coordinates in ${\C}^m$ in such a way that the projection $\pi$ onto the first $k$ coordinates realizes $\deg A$. Then for $p(x,y)=x$,the projection $(\pi\circ p)|_A$ has multiplicity $\deg A\cdot L^k$. Indeed, none of the functions $F_j$ can vanish identically on any irreducible component of $A$ due to the properness of $F$. Therefore, $F_j^{-1}(0)$ are pure $(k-1)$-dimensional (cf. \cite{D2}) and so we can find a point $z\in{\C}^k$ that is non-critical for $\pi|_A$ and does not lie in $\bigcup_j \pi(F_j^{-1}(0))$. Then the fibre $(\pi\circ p)^{-1}(z)\cap \Gamma$ has the maximal possible cardinality $\deg A\cdot L^k$. 

Take now $(x,y)\in\Gamma$ with $|x|\gg 1$, then 
$$
\mathrm{const.}|x|^L\leq |F(x)|=\max_j|F_j(x)|=|y|^L
$$
implies that for some $R>0$, 
$$
\Gamma\cap\{(x,y)\in{\C}^m\times{\C}^k\mid |x|\geq R\}\subset\{(x,y)\in{\C}^m\times{\C}^k\mid |x|\leq \mathrm{const.}|y|\}
$$
and so the projection $\varrho(x,y)=y$ realizes the degree of $\Gamma$. It is thus equal to $\mathrm{d}(F)$ and eventually,
$$
q^k\mathrm{d}(f)=\mathrm{d}(F)=\deg\Gamma\geq \deg A\cdot L^k=\deg A\cdot(q\ell)^k
$$
which gives the result sought for.
\end{proof}
\begin{rem}
Theorem \ref{krzywa} implies that the inequality in the last Theorem is strict.
\end{rem}

\section{Acknowledgements}

During the preparation of the first version of this paper, the second-named author was partially supported by Polish Ministry of Science and Higher Education grant  IP2011 009571. The present article has been much extended in comparison to the first version.

\end{document}